\documentclass[12pt]{article}
\usepackage{fullpage}
\usepackage{amssymb}
\usepackage{amsthm}
\usepackage[all,2cell]{xy}
\newtheorem{theorem}{Theorem}[section]
\newtheorem{corollary}[theorem]{Corollary}

\newtheorem{lemma}[theorem]{Lemma}

\begin{document}
\title{Projective normality of Weyl group quotients}
\author{S.S.Kannan, S.K.Pattanayak  \\  
\\ Chennai Mathematical Institute, Plot H1, SIPCOT IT Park,\\ Padur 
Post Office, Siruseri, Tamilnadu - 603103, India.\\
kannan@cmi.ac.in, santosh@cmi.ac.in} 

\maketitle
\date{}

\begin{abstract} In this note, we prove that for the standard representation $V$of
the Weyl group $W$ of a semi-simple algebraic group of type $A_n, B_n, C_n, D_n,
F_4$ and $G_2$ over $\mathbb C$, the projective variety $\mathbb P(V^m)/W$ 
is projectively normal with respect to the descent of $\mathcal O(1)^{\otimes |W|}$,
where $V^m$ denote the direct sum of $m$ copies of $V$. 
\end{abstract}
\hspace*{4.5cm}Keywords: line bundle, polarizations. 
\section{Introduction}
Let $G$ be a semi-simple algebraic group over $\mathbb C$ and $W$ be the 
corresponding Weyl group. Let $V$ be the standard
representation of $W$. By Noether's theorem the $\mathbb C$-algebra of 
invariants $\mathbb C[V^m]^W$ is finitely generated, (see [14]), where $V^m$ 
denote the direct sum of $m$ copies of $V$. So it is an interesting
problem to study GIT- quotient varieties $V^m/G=Spec(\mathbb C[V^m]^W)$ and 
$\mathbb P(V^m)/W$,(see [12] and [13]). Also, $\forall x \in \mathbb P(V^m)$, 
the isotropy $W_x$ acts trivially on the fiber of the line bundle 
$\mathcal O(1)^{\otimes |W|}$ at $x$. Hence, by Proposition (4.2), page 83 of 
[10], the line bundle $\mathcal O(1)^{\otimes |W|}$ descends to the quotient 
$\mathbb P(V^m)/W$, where $\mathcal O(1)$ denotes the ample generator of the 
Picard group of $\mathbb P(V^m)$. We denote it by $\mathcal L$. On the other 
hand, $V^m/W$ is normal. So, it is a natural question to ask if 
$\mathbb P(V^m)/W$ is projectively normal with respect to 
the line bundle $\mathcal L$. In [9] it is shown that the projective variety
$\mathbb P(V^m)/W$ is projectively normal with respect to the line bundle 
$\mathcal L$ for $m=1$. In this paper we show that $\mathbb P(V^m)/W$ is 
projectively normal with respect to the line bundle $\mathcal L$ for any $m$. 
 At the end of the paper we give a counter 
example showing that the result does not hold for symmetric groups over a field 
of positive characteristic. 

The layout of the paper is as follows: 

Section 2 consists of preliminary notations and definitions.

Section 3 consists of the main theorem, a corollary and a counter example.

\section{Preliminary notations and definitions}

Let $G$ be a semi-simple algebraic group of rank $n$ over $\mathbb C$. 
Let $T$ be a maximal torus of $G$. Let $N_G(T)$ be the normaliser of $T$ in 
$G$ and let $W=N_G(T)/T$ be the Weyl group of $G$ with respect to $T$. 
Consider the standard representation $V = Lie(T)$ of $W$. For every integer 
$m \geq 1$, the group $W$ acts on the algebra $\mathbb C[V^m]$ of polynomial 
functions on the direct sum $V^m:=V \oplus \cdots \oplus V$ of $m$ copies of 
$V$ via the diagonal action 

\[(wf)(v_1, \cdots ,v_m) := f(w^{-1}v_1, \cdots ,w^{-1}v_m), f \in \mathbb 
C[V^m], w \in W.\]    

If $m=1$ then the algebra  $\mathbb C[V]^W$ of invariants in one vector 
variable is generated by $n$ algebraic independent homogeneous invariants 
$f_1, f_2, \cdots ,f_n$ of degrees $d_1,d_2, \cdots ,d_n$ respectively such 
that $\prod_{i=1}^nd_i= |W|$ by a theorem of Chevalley-Serre-Shephard-Todd 
(see [1],[8],[16],[17]). We will
refer to such a system of generators of  $\mathbb C[V]^W$ as a system of basic
invariants. Explicit systems of basic invariants are known for each type of
irreducible Weyl groups $W$ (see [11]). 

Although a system of generators for  $\mathbb C[V^m]^W$ is not given 
independent of the Weyl group $W$, the classical approach to find a system of 
generators for $\mathbb C[V^m]^W$ is by the method of polarization. For type 
$A_n$, by a theorem of Hermann Weyl (see [19]) the algebra $\mathbb 
C[V^m]^{S_n}$ is generated by polarizations of the elementary symmetric 
polynomials. For the Weyl groups of types $B_n$, $C_n$ and $G_2$, the algebra 
$\mathbb C[V^m]^{W}$ is generated by polarizations of the basic invariants 
(see page 811 of [18]). However Wallach (see [18]) showed that the 
polarizations of the basic invariants do not generate $\mathbb C[V^m]^{W}$ for 
type $D_n$ and for $m \geq2$. 

Now we will recall the definition of polarizations of a polynomial from page 5 
of [19]. Let $f \in \mathbb C[V]^W$ be a homogeneous polynomial of degree $d$. 
For $ v_1,v_2, \cdots ,v_m \in V$ and $t_1,t_2, \cdots ,t_m$ are 
indeterminates, we consider the function $f(\sum_it_iv_i)$. Then 

\begin{equation}
f(\sum_it_iv_i) = \mathop{\bigoplus}_ {\alpha \in (\mathbb
Z^+)^m,|\alpha|=d}f_{\alpha}(v_1,\cdots ,v_m)t^{\alpha},
\end{equation}
where the $f_{\alpha} \in  \mathbb C[V^m]^W$ are multihomogeneous of the 
indicated degree $\alpha$. Here for $\alpha = (a_1,a_2, \cdots , a_m) \in 
(\mathbb Z^+)^m$, we have $t^{\alpha}= t^{a_1} \ldots t^{a_m}$ and 
$|\alpha|= a_1+ \ldots +a_m$. We call the polynomials $f_{\alpha}$, the 
{\it polarizations} of $f$.  

Polarizations of a polynomial can also be defined in terms of some linear
differential operators called the polarization operators. Choosing a basis for 
$V$ and writing $v_i= (x_{i1}, \cdots, x_{in})$ we define 

\[D_{ij}= \sum_{k=1}^{n}x_{ik}\frac {\partial}{\partial x_{jk}}.\]

The operators $D_{ij}$'s are called polarization operators. They commute with 
the action of $W$ on $\mathbb C[V^m]$ and applying successively operators 
$D_{ij}$ $(i > j)$ to $f \in \mathbb C[V]^W$ we obtain precisely (up to a 
constant) the polarizations of $f$ in any number of variables. 

Before ending this section we will recall the following.

Let $G$ be a finite group and $V$ be a finite dimensional, faithful 
representation of $G$ over an algebraically closed field of characteristic 
not dividing $|G|$. Let $\mathcal O(1)$ denote the ample generator of the Picard
group of $\mathbb P(V)$. Let $\mathcal L$ denote the descent of the line bundle
$\mathcal O(1)^{\otimes |G|}$ to the quotient $\mathbb P(V)/G$.

A polarized variety $(X, \mathcal L)$ where $L$ is a very ample line bundle is 
said to be projectively normal if its homogeneous coordinate ring 
$\oplus_{n \in \mathbb
Z_{\geq 0}}(H^0(X, \mathcal L^{\otimes n}))$ is integrally closed. For a reference,
see exercise 3.18, page 23 of [4].

The polarized variety ( $\mathbb P(V)/G, \mathcal L$) is 
 \[Proj (\oplus_{n \in \mathbb Z_{\geq 0}}(H^0(\mathbb P(V), \mathcal 
O(1)^{\otimes n|G|}))^G) = Proj (\oplus_{n \in \mathbb 
Z_{\geq 0}}(Sym^{n|G|}(V^{*}))^{G}).\]

For a reference, see Theorem 3.14 and page 76 of [12].

Now we will state a combinatorial lemma which can be applied to prove our main 
theorem.

Let \underline{$d$}=$(d_1,d_2,\cdots d_r) \in \mathbb N^{r}$ and
$N=(\prod_{i=1}^{r}d_{i})$. Consider the semigroup 

$M_{\underline{d}}=\{(m_1,m_2,\cdots m_r) \in \mathbb Z_{\geq
0}^{r}:\sum_{i=1}^{r}m_id_i\equiv 0 \,\mbox{mod}\,N\}$ and the set 

$S_{\underline{d}}=\{(m_1,m_2,\cdots m_r) \in \mathbb Z_{\geq
0}^{r}:\sum_{i=1}^{r}m_id_i= N\}$.

\begin{lemma}
 $M_{\underline{d}}$ is generated by $S_{\underline{d}}$ for \underline{$d$} $\in
\mathbb N^{r}$.
\end{lemma}
\begin{proof}

See lemma (2.1) of [9].

\end{proof}

\section{Main Theorem}
In this section we will prove our main theorem.

\begin{theorem}
Let $G$ be a semi-simple algebraic group of type $A_n,B_n,C_n,D_n,F_4$ or 
$G_2$. Let $W$ denote the corresponding Weyl group. Let $V$ be the standard 
representation of $W$. Then $\mathbb P(V^m)/W$ is projectively normal with 
respect to the line bundle
$\mathcal O(1)^{\otimes |W|}$. 
\end{theorem}

\begin{proof}

By a theorem of Chevalley-Serre-Shephard-Todd (see [1],[8],[16],[17]), the 
$\mathbb C$- algebra $\mathbb C[V]^W= (Sym(V^{*}))^W$ is a polynomial ring 
$K[f_1,f_2,\cdots, f_n]$ with each $f_i$ is a homogeneous polynomial of degree 
$d_i$ and $\prod_{i=1}^{n}d_i=|W|$. 

Let $R:=\oplus_{q\geq 0}R_q$; where $R_q:= (Sym^{q|W|}V^m)^W$. First we show that
$R$ is generated as a $\mathbb C$-algebra by $R_1$ by dealing with case by case.

${\underline {\bf Type \,\, A_n, B_n, C_n:}}$

For the diagonal action of the Weyl group on $V^m$, in type $A_n$ by H. Weyl 
(see pages 36-39 of [19]) and in type $B_n$ and $C_n$ by a theorem of Wallach 
(see page 811 of [18]), the algebra $\mathbb C[V^m]^W$ is generated by 
polarizations of the system of basic invariants $f_1,f_2,\cdots, f_n$.

For each $i \in \{1,2, \cdots ,n\}$, let $\{f_{ij}: j= 1,2, \cdots a_i\}$ 
denote the polarizations of $f_i$ where $a_i$ is a positive integer. Since 
the polarization operators $D_{ij}= \sum_{k=1}^{n}x_{ik}\frac {\partial}
{\partial x_{jk}}$ do not change the total degree of the original polynomial, 
we have 
\begin{equation}
 \mbox {degree of } f_{ij} =  \mbox {degree of } f_i = d_i, \,\, \forall \,\, j=
1,2, \cdots a_i.
\end{equation} 

Let us take an invariant polynomial $f \in (Sym^{q|W|}(V^m))^W$. Since $f_{ij}$'s
generate $\mathbb C[V^m]^W$ with out loss of generality we can assume $f$ is a
monomial of the form $\prod_{i=1}^n\prod_{j=1}^{a_i}f_{ij}^{m_{ij}}$. 

Since  $f = \prod_{i=1}^n\prod_{j=1}^{a_i}f_{ij}^{m_{ij}} \in (Sym^{q|W|}(V^m))^W$,
we have 
\[\sum_{i=1}^n\sum_{j=1}^{a_i}m_{ij}d_i = q|W|= q( \prod_{i=1}^nd_i)\]

Let $m_i = \sum_{j=1}^{a_i}m_{ij}$ then we have $\sum_{i=1}^nm_{i}d_i = q(
\prod_{i=1}^nd_i)$, and hence $(m_1, m_2, \cdots ,m_n)$ is in the semigroup
$M_{\underline d}=\{(m_1,m_2,\cdots m_r) \in \mathbb Z_{\geq
0}^{r}:\sum_{i=1}^{r}m_id_i\equiv 0 \,\mbox{mod}\,N\}$. 

By lemma (2.1), the semigroup $M_{\underline d}$ is generated by the set
$S_{\underline d}=\{(m_1,m_2,\cdots m_n) \in \mathbb Z_{\geq
0}^{n}:\sum_{i=1}^{r}m_id_i= \prod_{i=1}^nd_i\}$. So there exists $(m'_1,m'_2,
\cdots m'_n) \in \mathbb Z_{\geq 0}^{n}$ such that for each $i$
 \[m'_i < m_i \,\, \mbox {and} \,\,\sum_{i=1}^{r}m'_id_i= \prod_{i=1}^nd_i.\]

Again, since $m'_i < m_i= \sum_{j=1}^{a_i}m_{ij}$, for 
each $i$ and $j$ there exists $m'_{ij} \leq m_{ij}$ such that 
\[m'_i =  \sum_{j=1}^{a_i}m'_{ij}.\]

Then $g= \prod_{i=1}^n\prod_{j=1}^{a_i}f_{ij}^{m'_{ij}}$ is $W$-invariant and is in
$(Sym^{|W|}(V^m))^W$.

Let $f'= \frac {f}{g}$. Then $f' \in (Sym^{(q-1)|W|}(V^m))^W$ and so by 
induction on $q$, $f'$ is in the subalgebra generated by  $(Sym^{|W|}(V^m))^W$.

Hence $f=g.f'$ is in the subalgebra generated by $(Sym^{|W|}(V^m))^W$.   

${\underline {\bf Type \,\, D_n:}}$

Before proving the theorem for this case let us recall the action of the Weyl 
group of type $B_n$ and $D_n$ on the euclidean space $\mathbb R^n$. Let $W$ 
and $W'$ denote the Weyl group of type $D_n$ and $B_n$ respectively. Then $W'$ 
acts on $x=(x_1,x_2, \cdots ,x_n) \in \mathbb R^n$ by permutation of 
$x_1,x_2, \cdots ,x_n$
and the sign changes $x_i \rightarrow -x_i$ and the group $W$ acts on $x$ 
by permuting the coordinates and changes an even number of signs. Then it is 
clear that the group $W'$ is generated by the group $W$ and a reflection 
$\sigma$ defined by 
\[\sigma (x_1,x_2, \cdots ,x_{n-1},x_n)= (x_1,x_2, \cdots ,x_{n-1}, -x_n).\]

From [6] we  can take the polynomials  \[f_i = \sum_{k=1}^n x_k^{2i}, \,\, 
i= 1,2, \cdots , n-1\]
  and \hspace{4.4cm} $f_n = x_1.x_2 \ldots x_n$  \\
for the basic invariants of $\mathbb C[V]^W$.

For $\mathbb C[V]^{W'}$ we can take the basic invariants 
the polynomials \[f_i = \sum_{k=1}^n x_k^{2i}, \,\, i= 1,2, \cdots , n-1\]
  and \hspace{4.4cm} $f'_n = \sum_{k=1}^n x_k^{2n}$.

For odd $r \geq 1$, define the operator \[ P_r:= \sum_{k=1}^nx_{2k}^r \frac
{\partial}{\partial x_{1k}},\]  where $x_{1k},x_{2k}$ are standard coordinates of
$\mathbb R^{2n}$. The operator $P_r$ commutes with the diagonal action of $W$ 
and $W'$ on $\mathbb C[V^2]$ and preserves  $\mathbb C[V^2]^W$.

Now by [7] and [17] the algebra $C[V^2]^W$ is generated by the polarizations 
of the basic invariants $f_1, f_2, \cdots ,f_n$ and the polynomials 
\[P_{r_1}\cdots P_{r_l}(f_n) \,\, \,\, (r_i \geq 1 \,\, \mbox{ odd},\,\, 
\sum_{i=1}^lr_i \leq n-l).\] 

Note that the degree of the polynomial $P_{r}(f_n)$ is $n+r-1$ and so the 
degrees of the polynomials $P_{r_1}\cdots P_{r_l}(f_n), \,\, (r_i \geq 1 \,\, 
\mbox{ odd},\,\, \sum_{i=1}^lr_i \leq n-l)$ are \[n+(r_1+r_2+ \ldots +r_l)-l 
\leq 2n-2.\]

So $\mathbb C[V^2]^W$ is generated by homogeneous polynomials of degree 
$\leq 2n-2$.

Now we will prove the theorem for type $D_n$ by dealing with two cases.

\underline {Case -1:} n is even

In this case note that the degrees of the basic invariants $f_1, f_2, \cdots ,
f_n$ are all even. So the degrees of the polynomials $P_{r_1}\cdots 
P_{r_l}(f_n), \,\, (r_i \geq 1 \,\, \mbox{ odd},\,\, \sum_{i=1}^lr_i \leq n-l)$ 
are all even. Since the polarizations of the basic invariants have the same 
degrees as the basic invariants, we conclude that in this case the algebra 
$\mathbb C[V^2]^W$ is generated by homogeneous polynomials of even degrees 
less than or equal to $2n-2$.

Now for $m > 2$, by theorem (3.4) of [7], the algebra $\mathbb C[V^m]^W$ 
is generated by the polarizations of the generators of $\mathbb C[V^2]^W$. 
Again since the polarization operators do not change the degree of the 
original polynomial we conclude that the algebra $\mathbb C[V^m]^W$ is 
generated by homogeneous polynomials of even degrees same as the degrees of 
the basic invariants. So in this case we can employ the same proof as in the 
case of type $A_n,B_n$ and $C_n$.

\underline {Case -2:} n is odd

In this case since the degree of the basic invariant $f_n$ is 
odd and $r_i$'s are all odd, we have degrees of all the polynomials 
$P_{r_1}\cdots P_{r_l}(f_n), \,\, (r_i \geq 1 \,\, \mbox{ odd},\,\, 
\sum_{i=1}^lr_i \leq n-l)$ are odd.

Again, since for $m > 2$, the algebra $\mathbb C[V^m]^W$ is generated by the
polarizations of the generators of $\mathbb C[V^2]^W$, among the generators of
$\mathbb C[V^m]^W$ we have some odd degree invariants as well which are not
necessarily having the same degrees as the degree of $f_n$. 

Now, let us take one odd degree invariant $f \in \mathbb C[V^m]^W$ and write 
\[ f = \frac {f-\sigma (f)}{2}+ \frac {f+\sigma (f)}{2}\] where $\sigma$ is 
the reflection $(x_1,x_2, \cdots ,x_{n-1},x_n) \rightarrow (x_1,x_2, \cdots
,x_{n-1}, -x_n)$ defined before. 

Since the $W$ is a normal subgroup of the Weyl group $W'$ of type $B_n$  and 
$W'$ is generated by $W$ and $\sigma$, we have \[ \frac {f+\sigma (f)}{2} \in 
\mathbb C[V^m]^{W'}\]

Again, since $f$ is homogeneous of odd degree, the degree of 
$\frac {f+\sigma (f)}{2}$ is odd and hence equal to $0$ since 
$\mathbb C[V^m]^{W'}$ is generated by polarizations of the basic invaraints 
$f_1,f_2, \cdots ,f_{n-1},f'_n$ which are all of even degrees. Hence, for an 
odd degree invariant $f \in \mathbb C[V^m]^{W'}$, we have 
\[ \sigma (f) = -f.\]       
 
So for any $W$-invariant polynomials $f$ and $g$ of odd degrees we have $\sigma
(f.g) = fg$ and hence we conclude that $f^2, g^2$ and $fg$ are in $\mathbb
C[V^m]^{W'}$.

Now let us take a typical invariant monomial \[ f =
(\prod_{i}\prod_{j=1}^{a_i}f_{ij}^{m_{ij}})h_1^{l_1}h_2^{l_2}\ldots h_p^{l_p} \in
(Sym^{q|W|}V^m)^W\] where $g_{ij}$'s $\in \mathbb C[V^m]^W$ are the even degree
invariants of degrees $d_1,d_2, \cdots ,d_{n-1}$ obtained by taking the
polarizations of the even degree generators of $\mathbb C[V^2]^W$ and $h_i$'s 
$\in \mathbb C[V^m]^W$ are the odd degree invariants obtained by taking the 
polarizations of the odd degree generators of $\mathbb C[V^2]^W$.

Again since $h_i^2$ and $h_i.h_j$ are in $\mathbb C[V^m]^{W'}$, they are 
polynomials in $g_i$'s and the polarizations of the even degree basic 
invariant $f'_n$. So we may assume that $\sum_{i=1}^pl_i= 0$ or 1.

Suppose $\sum_{i=1}^pl_i= 1$, then f  is of the form \[f
=(\prod_{i}\prod_{j=1}^{a_i}f_{ij}^{m_{ij}}).h \in \mathbb C[V^m]^{W},\] 

where $h$ is of odd degree, say $t$. So we have 
\[\sum_{i}\sum_{j=1}^{a_i}m_{ij}d_i +t = q.|W|.\]

This is not possible since $d_i$'s are all even and $|W|$ is even. So we 
conclude that $\sum_{i=1}^pl_i= 0$ and hence $f$ is of the form 
$g_1^{m_1}g_2^{m_2}\ldots g_r^{m_r}$ where $g_i$'s are all of even degrees less 
than equal to $2n$. So in this case we can proceed the proof as in the case 
of Type $A_n, B_n$ and $C_n$.

${\underline {\bf Type F_4 \,\, and \,\, G_2:}}$

Since the cardinality of the Weyl group of Type $G_2$ is $12 = 2^2.3$ and the
cardinality of the Weyl group of Type $F_4$ is $1152 = 2^7.3^2$, by Burnside's 
$p-q$ theorem ( see page 247 of [15]), they are solvable. Hence the result is 
true for each case by proposition (1.1) of [9].

Now the proof of the theorem follows from exercise $5.14(d)$, Chapter II of 
[4]. 
\end{proof}

We deduce the following Result of Chu-Hu-Kang (see [2]) as a consequence of Theorem 3.1.

\begin{corollary}
Let $G$ be a finite group and $U$ be any finite dimentional representation of 
$G$ over $\mathbb C$. Let $\mathcal L$ denote the descent of 
$\mathcal O(1)^{\otimes n!}$. Then $\mathbb P(U)/G$ is projectively normal 
with respect to $\mathcal L$.

\end{corollary}

\begin{proof}
Let $G = \{g_1, g_2, \cdots , g_n\}$ and let $\{u_1, u_2, \cdots , u_k\}$ be a 
basis of $U$. Let $V$ be the natural representation of the permutation group 
$S_n$. Let $\{x_1, x_2, \cdots , x_n\}$ be a basis of $V$; then the set 
$\{x_{11}, \cdots ,x_{n1}, \cdots , x_{1k}, \cdots , x_{nk}\}$ is a basis of 
$V^k$. 

Consider the Cayley embedding $G \hookrightarrow S_n$, $g \mapsto (g_j :=gg_i)$.
 Then 
\[ \eta : Sym (V^k) \rightarrow Sym (U), \,\, \,\, x_{il} \mapsto g_i(u_l)\] 
is a $G$-equivariant and degree preserving algebra epimorphism. 

Now we will use Noether's original argument (see page 2 of [14]) to show that
 the 
restriction map \[\tilde{\eta}: Sym (V^k)^{S_n} \rightarrow Sym (U)^G\] is 
surjective. For any $f = f(u_1, \cdots , u_k) \in  Sym (U)^G$, we define 
\[f': = \frac {1}{n}(f(x_{11}, x_{12}, \cdots ,x_{1k})+ \ldots +f((x_{n1}, 
x_{n2}, \cdots ,x_{nk})) \in Sym (V^k)^{S_n}.\] Then we have \[\tilde{\eta}(f')=
\frac{1}{n}(f(g_1(u_1), g_1(u_2), \cdots , g_1(u_k))+ \ldots + f(g_n(u_1), g_n(u_2), 
\cdots , g_n(u_k)))\]
\[= \frac{1}{n}(g_1f(u_1, u_2, \cdots , c_k)+ \ldots g_nf(u_1, u_2, \ldots , u_k))= f\]

Hence,$\tilde{\eta}(f')=f$ and $\tilde{\eta}$ is surjective. So the corollary 
follows from theorem (3.1).

\end{proof}

${\underline {\bf A \,\, counter \,\, example:}}$ 

Let $F$ be a field of characteristic $p \neq 2$ and $V$ be the natural 
representation of the permutation group $G = S_{p^s}, s \geq 2$ over $F$. 
Consider $U= \underbrace{V \oplus V \oplus \ldots \oplus V}_{(p^s)! \,\, 
\mbox{copies}}$. Then by therem (4.7) of [3] we have \[\beta(U, G) = max \{p^s, 
(p^s)!(p^s-1)\} = (p^s)!(p^s-1),\] where $\beta(U, G)$ denote the Noether 
number, which can be defined as the minimal number $t$, such that the algebra
$Sym(U^*)^G$ of invariants can be generated by finitely many elements of degree 
at most $t$. 

So there exists a homogeneous polynomial $ f \in (Sym^d U^*)^G;\,\,  
d = (p^s)!(p^s-1)$ which is not in the subalgebra generated by 
$(Sym^m U^*)^G; \,\, m \leq (d-1)$. 

Hence, $ f \in R_{(p^s-1)}= (Sym^{(p^s-1)|G|} U^*)^G$ but not in the subalgebra 
generated by $R_1= (Sym^{|G|} U^*)^G$. Thus, projective normality does not 
hold in this case.

{\it Remark 1:} We couldn't find any reference for the generators of 
$\mathbb C[V^m]$ for type $E_6, E_7, E_8$. We are now working on it. Due to 
time constraint, we will write it in the future work.

{\it Remark 2:} We beleive that from theorem (3.1), we will be able to prove 
the Projective normality result for any finite dimensional representation 
of any Weyl group. We are working on this problem.



\end{document}